\documentclass[sn-mathphys]{sn-jnl}

\DeclareMathOperator{\Diag}{Diag}
\DeclareMathOperator{\coker}{coker}
\DeclareMathOperator{\imagen}{im}
\DeclareMathOperator{\tr}{tr}
\DeclareMathOperator{\adj}{adj}

\DeclareMathOperator{\HH}{\mathrm{H}}
\DeclareMathOperator{\SL}{\mathrm{SL}}

\newcommand{\id}{\mathrm{id}}

\newtheoremstyle{mydef}
{}		
{}		
{}		
{}		
{\scshape}	
{. }		
{ }		
{\thmname{#1}\thmnumber{ #2}\thmnote{ #3}}	

\theoremstyle{thmstyleone}%

\newtheorem{theorem}{Theorem}[section]
\newtheorem*{theorem*}{Theorem}
\newtheorem{lemma}{Lemma}[section]
\newtheorem{notation}{Notation}[section]
\newtheorem{example}{Example}[section]%
\newtheorem{definition}{Definition}[section]
\newtheorem*{conjecture*}{Conjecture}
\newtheorem*{KKP}{KKP Conjecture}
\theoremstyle{thmstyletwo}%

\theoremstyle{thmstylethree}%
\newtheorem{remark}{Remark}[section]%

\begin{document}
	
\title[KKP conjecture for minimal  adjoint orbits]{The Katzarkov--Kontsevich--Pantev conjecture for minimal  adjoint orbits}
	

\author[1]{\fnm{Edoardo} \sur{ Ballico}}\email{ballico@science.unitn.it}

\author*[2]{\fnm{Elizabeth} \sur{Gasparim}}\email{etgasparim@gmail.com}

\author[3]{\fnm{Francisco} \sur{Rubilar}}\email{francisco.rubilar@ucsc.cl}

\author[4]{\fnm{Luiz A.} \sur{B. San Martin}}\email{smartin@ime.unicamp.br}

\affil[1]{\orgdiv{Dept. Mathematics}, \orgname{Univ. of Trento}, \orgaddress{\street{I-38050}, \city{Povo}, \postcode{38123}, \state{Trento}, \country{Italy}}}

\affil[2*]{\orgdiv{Dept. Mathematics}, \orgname{Univ. Cat\'olica del Norte}, \orgaddress{\street{Angamos 610}, \postcode{1270236}, \state{Antofagasta}, \country{Chile}}}

\affil[3]{\orgname{Univ. Cat\'olica de la Sant\'isima Concepci\'on}, \orgaddress{\street{Arauco 449}, \postcode{3820572}, \state{Chill\'an}, \country{Chile}}}

\affil[4]{\orgdiv{Dept. Mathematics}, \orgname{Univ. Estadual de Campinas}, \orgaddress{\street{Cidade Universit\'aria}, \city{Campinas}, \postcode{13083-856}, \country{Brazil}}}
 	
	
	\abstract{We consider  minimal semisimple 
		adjoint orbits as  Landau--Ginzburg models and prove that they satisfy the conjecture of  Katzarkov--Kontsevich--Pantev 
	 about new Hodge theoretical  invariants.}
	

	
	
	
	\maketitle
	\newpage
	
	\section{KKP conjecture and our result}

	The Katzarkov--Kontsevich--Pantev  (KKP) conjecture is a numerical prediction 
	expected to follow from the Homological Mirror Symmetry conjecture of M. Kontsevich, a master conjecture predicting a wide range of 
	categorical equivalences which so far have been  established 
	in only a  few cases. 
	
	In  \cite{KKP}  three types of new Hodge theoretical invariants were defined:
	$$
	f^{p,q}(Y,w),\ \ h^{p,q}(Y,w),\ \ i^{p,q}(Y,w)
	$$
	for tamely compactifiable  Landau--Ginzburg (LG) models $w\colon Y\to \mathbb C$;
	we recall their definitions  in section \ref{invariants}.
	 Katzarkov--Kontsevich--Pantev proved in \cite{KKP}  that these numbers satisfy the identities
	\begin{equation}\label{sums}
		h^m(Y,Y_b;\mathbb C)=\sum _{p+q=m}i^{p,q}(Y,w)=\sum _{p+q=m}h^{p,q}(Y,w)=\sum _{p+q=m}f^{p,q}(Y,w)
	\end{equation}
	where $Y_b$ is a smooth fibre of $w$, and  conjectured the equality of the 3 invariants.
\begin{KKP}
\cite{KKP, LP} Assume that $(Y,w)$ is tamely compactified  Landau--Ginzburg. Then for every $p,q$ there are  equalities
		\begin{equation*}\label{equat-of-indiv-hodge-numbers}
			h^{p,q}(Y,w)=f^{p,q}(Y,w)=i^{p,q}(Y,w) .
		\end{equation*}
\end{KKP}
For  $Y$  a specific rational surface with a map $w\colon Y\to \mathbb C$ such that the generic fibre is an elliptic curve, 
	 Lunts and Przyjalkowski \cite{LP} proved the equality $f^{p,q}(Y,w)=h^{p,q}(Y,w)$ and gave an example  where~$i^{p,q}(Y,w)\neq h^{p,q}(Y,w)$. 	 Shamoto \cite{Sh} gave sufficient conditions for a tamely compactifiable LG models to satisfy $f^{p,q}(Y,w)=h^{p,q}(Y,w)$.
	 There remains open the  question of 
	what LG models satisfy this KKP conjecture. 	
	The goal of this paper is to provide a family of examples of LG models coming from Lie theory that do satisfy the KKP conjecture 
	in the sense that the 3 invariants coincide on weakly tame compactifications. 	\\

 We note that  Katzarkov--Kontsevich--Pantev  posed a second conjecture in \cite{KKP}, namely
 they predicted that if $(Y, w)$ is the LG model mirror to $X$, then $f^{p,q}(Y,w) = h^{p, n-q}(X)$ where $n=\dim X=\dim Y$. This conjecture was previously discussed by Gross--Katzarkov--Ruddat \cite{GKR} in the context of mirror symmetry for varieties of general type.
 Harder \cite{Ha} described how to compute  the Hodge numbers $f^{p,q}(X)$  via classical mixed Hodge  
	 theory, proving $h^{p,q}(X)= f^{3-q,p}(Y,w)$ for a crepant resolution of a Gorenstein Fano threefold $X$ 
	 and its  LG mirror $(Y,w)$. 
  Cheltsov and Przyjalkowski  \cite{CP} showed that this second conjecture  holds  for the Hodge numbers of 
  Fano threefolds and their mirror LG models.\\

 In this work we consider the KKP 
  conjecture  about the equality of the 3 invariants defined by Katzarkov--Kontsevich--Pantev, 
  providing examples where  the 3 numbers coincide for a family of Landau--Ginzburg models $(Y,w)$  admitting
   {\it weakly tame} 
  compactifications $((Z,f),D_Z)$.  We call our compactifications weakly tame because, 
even though the  the divisor $D_Z$  is the sum of the fibre
at infinity $D^v$ and the horizontal divisor at infinity $D^h$, only  multiples of  these divisors give the anti-canonical class,
failing the  strict requirement of multiplicity 1  in Definition \ref{def-3}. The appearance of these multiplicities 
is unavoidable here, because our compactifications $Z$ are Fano.

		We consider the symplectic Lefschetz fibrations constructed in \cite{GGSM1}, defined over semisimple adjoint 
	orbits via the height function from Lie theory. The novelties of the cases that we present are: these adjoint orbits are 
	affine, non-toric, Calabi--Yau varieties, which occur in arbitrarily high dimensions.  Indeed, the minimal semisimple orbit of 
	$\mathfrak{sl}(n+1)$ has dimension $2n$. Thus, we present the first explicit family of arbitrarily large 
	Calabi--Yau varieties where the KKP conjecture 
	holds true in the sense that all 3 invariants   coincide.

	Let $\mathfrak g$ be a complex semisimple Lie algebra with Lie group $G$,  and $\mathfrak h$ the Cartan subalgebra. 
	Consider the adjoint orbit  $\mathcal O(H_0)$  of an element $H_ 0 \in \mathfrak h$, 
	that is, $$\mathcal O(H_0) := \{\mathrm{Ad} (g) H_0, g \in G\}\text{.}$$ Let $H  \in \mathfrak{h}_{\mathbb{R}}$ be a regular element, and 
	$\langle \cdot , \cdot\rangle$ the Cartan--Killing form. 
	Then \cite[Thm.\thinspace 2.2]{GGSM1} shows that the height function 
	\begin{equation}\label{height} \begin{array}{rll} f_H\colon  \mathcal O(H_0) & \rightarrow & \mathbb C\cr
			X & \mapsto & \langle H,X \rangle
		\end{array}
	\end{equation}
	gives the orbit the structure of  a symplectic Lefschetz fibration.

	Here we consider the case of    $G=\mathrm{SL}(n+1,\mathbb{C})$ and focus on
	the adjoint orbit passing through  $H_0=\Diag(n,-1,
	\ldots,-1)$. The diffeomorphism type is then $\mathcal{O}(H_0)  \simeq T^* \mathbb P^{n}$.
	Among all choices of elements $H_0\in \mathfrak h \subset {\mathfrak sl}(n+1, \mathbb C) $, the choice 
	$H_0=\Diag(n,-1,\ldots,-1)$ produces the homogeneous manifold of smallest dimension, 
	for this reason we set the following terminology:
	
	\begin{definition}\label{orbit}
		Let $\mathcal O_n$ denote the adjoint orbit of  $H_0=\Diag(n,-1,\ldots,-1)$ in
		${\mathfrak sl}(n+1,\mathbb{C})$, we
		call it the {\it minimal orbit}.
	\end{definition}
	For every $n\geq 2$ 
	and for every choice of regular element $H \in \mathfrak h$, we prove:
	
	\begin{theorem*}\ref{main}
		The LG model $(\mathcal O_n,f_H)$ admits a weakly tame compactification and satisfies the KKP conjecture.
	\end{theorem*}
	
	Since the 3 Hodge theoretical invariants $f^{p,q}(Y,w), h^{p,q}(Y,w), i^{p,q}(Y,w),$
	of our LG models coincide, we may depict diamonds containing their common  values, 
	that we  name   {\it KKP diamonds}. For our examples, 
	these turn out to  have a unique nonzero entry. Thus, our Lie theoretical 
	constructions  give rise to LG models where computations of Hodge theoretical invariants turn out rather simple, 
	which is arguably the best feature  of our choice of   LG models. It is at this moment unknown how KKP invariants 
	behave for 
	LG models on more general adjoint orbits.
	
 \section{Landau--Ginzburg Hodge numbers}\label{invariants}
 This section is  just a summary  of parts of the nicely written text of  \cite{LP}.

 \begin{definition} \label{def-1}\emph{A Landau--Ginzburg model} is a pair $(Y,w)$, where
 	
 	\begin{enumerate}
 		\item $Y$ is a smooth complex quasi-projective variety with trivial canonical bundle $K_Y$;
 		\item $w\colon Y\to \mathbb C$ is a morphism with a compact critical locus $crit(w)\subset Y$.
 	\end{enumerate}
 	
 \end{definition}

 \begin{definition}\cite[Def.\thinspace 3]{LP} \label{def-3}\emph{A tame compactified Landau--Ginzburg model} is the data $((Z,f),D_Z)$, where
 	
 	\begin{enumerate}
 		\item $Z$ is a smooth projective variety and $f\colon Z\to \mathbb P ^1$ is a flat morphism.
 		
 		\item $D_Z=(\cup _i D^h_i)\cup (\cup _jD_j^v)$ is a reduced normal crossings divisor such that
 		
 		\begin{itemize}
 			\item[(i)] $D^v=\cup _jD^v_j$ is a scheme theoretical pole divisor of $f$, i.e. $f^{-1}(\infty)=D^v$. In particular 
			$ord _{D^v_j}(f)=-1$ for all $j$;
 			
 			\item[(ii)] each component $D_i^h$ of $D^h=\cup _iD^h_i$ is smooth and horizontal for $f$, i.e. $f\vert _{D^h_i}$ is a flat morphism;
 			
 			\item[(iii)] The critical locus $crit(f)\subset Z$ does not intersect $D^h$.
 			
 		\end{itemize}

 		\item $D_Z$ is an anticanonical divisor on $Z$.
 		
 		\noindent One says that $((Z,f),D_Z)$ is \emph{a compactification of the Landau--Ginzburg model}
 		$(Y,w)$ if in addition the following holds:
 		
 		\item $Y=Z\setminus D_Z$, $f\vert _Y=w$. 
 	\end{enumerate}
 	
 \end{definition}
 
 Assuming that we are given a Landau--Ginzburg model $(Y,w)$ with a tame compactification~$((Z,f),D_Z)$ as above,
  in\thinspace\cite{KKP} the authors define the  numbers  $i^{p,q}(Y,w)$, $h^{p,q}(Y,w)$, $f^{p,q}(Y,w)$.
   We recall their definitions, and we will use them in our more general situation of weakly tame compatification.
 We denote by $n=\dim Y=\dim Z$ the (complex) dimension of~$Y$ and $Z$. Choose a point $b\in  \mathbb C$ which is near $\infty$ 
 and such that the fibre $Y_b=w^{-1}(b)\subset Y$ is smooth.

 \subsection{$f^{p,q}(Y,w)$}\label{subs-def-fpq} Recall  the logarithmic de Rham complex~$\Omega^\bullet _Z(log\, D_Z)$.
 Namely, $\Omega ^s _Z(log\,D_Z)=\wedge ^s \Omega^1 _Z(log\,D_Z )$ and $\Omega^1 _Z(log\,D_Z )$ is 
 a locally free~$\mbox{$\mathcal O _Z$-module}$ generated locally by
 $$\frac{dz_1}{z_1},\ldots,\frac{dz_k}{z_k},dz_{k+1},\ldots,dz_n$$
 if $z_1\cdot \ldots\cdot z_k=0$ is a local equation of the divisor $D_Z$. Hence in particular $\Omega ^0 _Z(log\,D_Z)=\mathcal O _Z$.
 
 The numbers $f^{p,q}(Y,w)$ are defined using the  subcomplex $\Omega^\bullet _Z(log\,D_Z ,f)\subset \Omega^\bullet _Z(log\,D_Z)$ of~$f$-{\it adapted forms}, which we recall next.
 
 \begin{definition} For each $a\geq 0$ define \emph{a sheaf $\Omega ^a _Z(log\,D_Z ,f)$ of $f$-adapted logarithmic forms} as a subsheaf of $\Omega ^a _Z(log\,D_Z)$ consisting of forms which stay logarithmic after multiplication by $df$. Thus
 	$$\Omega ^a _Z(log\,D_Z ,f)=\{\alpha \in \Omega ^a _Z(log\,D_Z)\ \vert \ df\wedge \alpha \in \Omega ^{a+1} _Z(log\,D_Z )\},$$
 	where one considers $f$ as a meromorphic function on $Z$ and $df$ is viewed as a meromorphic~$\mbox{1-form}$.
 \end{definition}
 
 \begin{definition}\cite[Def.\thinspace 3.1]{KKP} \label{fpq}\emph{The Landau--Ginzburg Hodge numbers} $f^{p,q}(Y,w)$ are defined as follows:
 	$$f^{p,q}(Y,w)=\dim \HH^p(Z,\Omega ^q _Z(log\,D_Z ,f)).$$
 \end{definition}
 
 \subsection{$h^{p,q}(Y,w)$} Let $N\colon V\to V$ be a nilpotent operator on a finite dimensional vector space $V$ such that $N^{m+1}=0$. Such data defines a canonical {\it weight filtration centered at $m$}, $W=W_\bullet (N,m)$ of $V$
 $$0\subset W_0(N,m)\subset W_1(N,m)\subset \ldots\subset W_{2m-1}(N,m)\subset W_{2m}(N,m)=V$$
 with the properties
 \begin{enumerate}
 	\item $N(W_i)\subset W_{i-2}$,
 	\item the map $N^l\colon gr ^{W,m}_{m+l}V\to gr ^{W,m}_{m-l}V$ is an isomorphism for all $l\geq 0$.
 \end{enumerate}
 
 Let $S^1\simeq C\subset \mathbb P ^1$ be a loop passing through the point $b$ that goes once around $\infty$
 in the counter clockwise direction in such a way that there are no singular points of $w$ on or inside~$C$. It gives 
 the monodromy transformation
 \begin{equation*} \label{monodromy-1}T\colon \HH^\bullet (Y_b)\to \HH^\bullet(Y_b)
 \end{equation*}
 with corresponding monodromy transformation on the relative cohomology
 \begin{equation}\label{monodromy-2} T\colon \HH^\bullet (Y,Y_b)\to \HH^\bullet (Y,Y_b).
 \end{equation}
 in such a way that the sequence
 \begin{equation*}\label{long-exact-seq}
 	\ldots\to \HH^m(Y,Y_b)\to \HH^m(Y)\to \HH^m(Y_b)\to \HH^{m+1}(Y,Y_b)\to \ldots
 \end{equation*}
 is $T$-equivariant, where $T$ acts
 trivially on $\HH^\bullet (Y)$. 
 
 Since we assume that the infinite 
 fibre $f^{-1}(\infty)\subset Z$ is a reduced divisor with normal crossings, by Griffiths--Landman--Grothendieck 
 Theorem (see\thinspace\cite{Ka}) the operator $T\colon \HH^m (Y_b)\to \HH^m (Y_b)$ is unipotent and 
 $\left(T-\id\right)^{m+1}=0$. It follows that the transformation \eqref{monodromy-2} is also unipotent.

 \begin{definition}\label{hpq}
\cite[Def.\thinspace 3.2]{KKP} $\label{wrong-def}h^{p,q}(Y,w)=\dim gr _{p}^{W,p+q}\HH^{p+q}(Y,Y_b).$
\end{definition}

 \begin{remark}
 Definition \thinspace\ref{hpq} differs from \cite[Def.\thinspace 8]{LP}
 by the indices of the grading. 
  Denoting by $N$ the logarithm of the transformation \eqref{monodromy-2}, which is assumed to 
  be  a nilpotent operator on $\HH^\bullet (Y,Y_b)$ with $N^{m+1}=0$, \cite{LP} 
call the Landau--Ginzburg model $(Y,w)$  of {\it Fano type} if the operator~$N$ on 
 the relative cohomology $\HH^{n+a} (Y,Y_b)$ has the  properties:
  $N^{n-\vert a\vert }\neq 0, \quad \text{and}\quad 
 		 N^{n-\vert a \vert +1}=0,$
and define:
 	$$h^{p,n-q}(Y,w)=\dim gr _{2(n-p)}^{W,n-a}\HH^{n+p-q}(Y,Y_b)\ \ \text{if $a=p-q\geq 0$},$$
 	$$h^{p,n-q}(Y,w)=\dim gr _{2(n-q)}^{W,n+a}\HH^{n+p-q}(Y,Y_b)\ \ \text{if $a=p-q< 0$}.$$
 \end{remark}
 
 We shall not  worry about  details of the grading here; it is trivial in our examples, 
 since we have no critical points on the  fibre at infinity. 
 Such a possibility is  explicitly 
 mentioned in \cite[Def.\thinspace3.2(iii)]{KKP}.
 
 \subsection{$i^{p,q}(Y,w)$} For each $\lambda \in \mathbb C$ one
 has the corresponding sheaf~$\phi _{w-\lambda}\mathbb C _Y$ of vanishing cycles for the fibre $Y_\lambda$.
 The sheaf $\phi _{w-\lambda}\mathbb C _Y$ is supported on the fibre $Y_\lambda$ and is equal to zero if $\lambda $ 
 is not a critical value of $w$. From the works of Schmid, Steenbrink, 
 and Saito it is classically known that the constructible complex $\phi _{w-\lambda}\mathbb C _Y$ carries a structure of a mixed Hodge 
 module and so its hypercohomology inherits a mixed Hodge structure. For a mixed Hodge module $S$ we will denote by 
 $i^{p,q}S$ the $(p,q)$ Hodge numbers of the $p+q$ weight graded piece $gr ^W_{p+q}S$.
 
 \begin{definition}\cite[Def.\thinspace 3.4]{KKP}\label{ipq}
 \begin{enumerate}
 		\item
 	Assume that the horizontal divisor $D^h\subset Z$ is empty, i.e. assume that the map $w\colon Y\to \mathbb C$ is proper. Then
 		\emph{the Landau--Ginzburg Hodge numbers}~$i^{p,q}(Y,w)$ are defined as follows:
 		$$
 		i^{p,q}(Y,w)=\sum _{\lambda \in \mathbb C}\sum _ki^{p,q+k}\mathbb H ^{p+q-1}(Y_\lambda ,
 		\phi _{w-\lambda}\mathbb C _Y).
 		$$
 		
 		\item
 		In the general case denote by $j\colon Y\hookrightarrow Z$ the open embedding and define
 		similarly
 		$$
 		i^{p,q}(Y,w)=\sum _{\lambda \in \mathbb C}\sum _ki^{p,q+k}\mathbb H ^{p+q-1}(Y_\lambda ,
 		\phi _{w-\lambda}{\bf R}j_{*}\mathbb C _Y).
 		$$
 	\end{enumerate}
 \end{definition}

\section{Lie theoretical compactification }\label{Liecompact}
Let $\mathfrak{g}$ be a noncompact semisimple  Lie algebra (not necessarily complex) with group $G$.
A compactification of $\mathcal O(H_0)$ to a product of flags $ F_{\Theta } \times F%
_{\Theta ^{\ast }}$ is described in 
\cite[Sec.\thinspace 3]{GGSM2}. We now describe the orbits of the diagonal action of $G$ in this product. 
For the case considered here, the one of minimal orbits, we will have $ F_{\Theta } = \mathbb P^n\simeq
Gr(n,n+1) = F_{\Theta ^{\ast }}$, see example \ref{productPn}.

Let $\Sigma $ be a system of simple roots of 
$\left( \mathfrak{g},\mathfrak{a}\right) $ (where $%
\mathfrak{g}=\mathfrak{k}\oplus \mathfrak{a}\oplus \mathfrak{n}^{+}$ is an 
Iwasawa decomposition) and  $\Theta \subset \Sigma $ a subset of roots, cf. example \ref{roots}. Choose $%
H_{\Theta }$ defined by $\Theta =\{ \alpha \in \Sigma :\alpha \left(
H_{\Theta }\right) =0\}$, then set 
\begin{equation*}
	\mathfrak{n}_{H_{\Theta }}^{+}=\sum_{\alpha \left( H_{\Theta }\right) >0}%
	\mathfrak{g}_{\alpha } \mathrm{,}\qquad\qquad \mathfrak{n}_{H_{\Theta
	}}^{-}=\sum_{\alpha \left( H_{\Theta }\right) <0}\mathfrak{g}_{\alpha }
\end{equation*}%
and take the parabolic subalgebra
\begin{equation*}
	\mathfrak{p}_{\Theta }=\sum_{\lambda \geq 0}\mathfrak{g}_{\lambda }=%
	\mathfrak{z}_{\Theta }\oplus \mathfrak{n}_{H_{\Theta }}^{+}
\end{equation*}%
where $\lambda $ varies over the eigenvalues of  $\mathrm{ad}\left( H_{\Theta
}\right) $ and $\mathfrak{z}_{\Theta }$ is the centralizer of  $H_{\Theta }$%
. The dual of $\Theta$  is  by definition $$\Theta ^{\ast }:=-w_{0}\left( \Theta \right) \subset \Sigma $$
where $w_{0}$ is the main involution of the  Weyl group $\mathcal{W%
}$. Set%
\begin{equation*}
	\mathfrak{q}_{\Theta ^{\ast }}=\sum_{\lambda \leq 0}\mathfrak{g}_{\lambda }=%
	\mathfrak{z}_{\Theta }\oplus \mathfrak{n}_{H_{\Theta }}^{-},
\end{equation*}%
the   parabolic subalgebra of $\mathfrak g$ conjugate to $%
\mathfrak{p}_{\Theta ^{\ast }}$. In fact, $\mathfrak{q}_{\Theta ^{\ast }}=%
\mathrm{Ad}\left( \overline{w}_{0}\right) \left( \mathfrak{p}_{\Theta ^{\ast
}}\right) $ where $\overline{w}_{0}$ is a representative of the main involution
$w_{0}$ in $\mathrm{Norm}_{G}\left( \mathfrak{a}%
\right) $, and this is precisely the reason to consider here the dual flag  $F_{\Theta ^{\ast }}$.

The parabolic subgroups  $P_{\Theta }$ and $P_{\Theta ^{\ast }}$ are the 
normalizers of  $\mathfrak{p}_{\Theta }$ and $\mathfrak{p}_{\Theta ^{\ast
}}$ respectively. Their flags are $F_{\Theta }=G/P_{\Theta }$
and $F_{\Theta ^{\ast }}=G/P_{\Theta ^{\ast }}$.

Denote by $b_{\Theta }=1\cdot P_{\Theta }$ the origin of  $F_{\Theta
}=G/P_{\Theta }$ and  by $b_{\Theta ^{\ast }}=1\cdot P_{\Theta ^{\ast }}$ the origin of 
$F_{\Theta ^{\ast }}=G/P_{\Theta ^{\ast }}$. If $w\in \mathcal{W}$
then $wb_{\Theta }$ (respectively $wb_{\Theta ^{\ast }}$) denotes the 
image of  $b_{\Theta }$  by $w$ (actually, the image $\overline{w}b_{\Theta
}$ of any representative  $\overline{w}\in \mathrm{Norm}_{G}\left( 
\mathfrak{a}\right) $ of $w$).

The diagonal action is given by $g\left( x,y\right) =\left(
gx,gy\right) $, $g\in G$, $x\in F_{\Theta }$ and $y\in F%
_{\Theta ^{\ast }}$.
The following statements describe the diagonal action and its properties.

\begin{enumerate}
	\item Orbits of the diagonal action have the form $G\cdot
	\left( b_{\Theta },wb_{\Theta ^{\ast }}\right) $ with $w\in \mathcal{W}$.
	
	In fact, given $\left( x,y\right) \in F_{\Theta }\times F%
	_{\Theta ^{\ast }}$ there exists $g\in G$ such that $x=gb_{\Theta }$. Therefore, $%
	\left( x,y\right) $ is in the orbit of  $\left( b_{\Theta },z\right) $
	for some $z\in F_{\Theta ^{\ast }}$.
	
	On the other hand $F_{\Theta ^{\ast }}$ is the union of  orbits 
	$N^{+}\cdot wb_{\Theta ^{\ast }}$, $w\in \mathcal{W}$. Thus, $z\in
	N^{+}\cdot wb_{\Theta ^{\ast }}\subset P_{\Theta }\cdot wb_{\Theta ^{\ast }}$
	for some $w\in \mathcal{W}$. This shows that any 
	$\left( x,y\right) $ belongs to an orbit $G\cdot \left( b_{\Theta },wb_{\Theta ^{\ast }}\right) $
	for some $w\in \mathcal{W}$.
	
	Note that  for different  $w\in \mathcal{W}$ it might happen that the
	orbits $G\cdot \left( b_{\Theta },wb_{\Theta ^{\ast }}\right) $
	coincide.
	
	\item Dualizing, it follows that  orbits of the diagonal action are of the form
	$G\cdot \left( wb_{\Theta },b_{\Theta ^{\ast }}\right) $ with $w\in 
	\mathcal{W}$. The two descriptions are equivalent, since $\left(
	wb_{\Theta },b_{\Theta ^{\ast }}\right) $ and $\left( b_{\Theta
	},w^{-1}b_{\Theta ^{\ast }}\right) $ belong to the same orbit.
	
	\item The two previous items show that the diagonal action has only
	a finite number of orbits.
	
	The orbits $G\cdot \left( b_{\Theta },wb_{\Theta ^{\ast
	}}\right) $ are in bijection with the orbits $P_{\Theta
	}\cdot wb_{\Theta ^{\ast }}$, which are all the orbits of  $P_{\Theta
	}$ in $F_{\Theta ^{\ast }}$. They are also in bijection with 
	the orbits $P_{\Theta ^{\ast }}\cdot wb_{\Theta }$ which are
	the orbits of  $P_{\Theta ^{\ast }}$ in $F_{\Theta }$. In fact, if
	$\left( b_{\Theta },wb_{\Theta ^{\ast }}\right) $ and $\left( b_{\Theta
	},w_{1}b_{\Theta ^{\ast }}\right) $ belong to the same orbit, then there
	exists $g\in G$ such that $g\left( b_{\Theta },wb_{\Theta ^{\ast }}\right)
	=\left( b_{\Theta },w_{1}b_{\Theta ^{\ast }}\right) $. This means that $%
	gb_{\Theta }=b_{\Theta }$, that is, $g\in P_{\Theta }$. Consequently $%
	wb_{\Theta ^{\ast }}=gw_{1}b_{\Theta ^{\ast }}$ with $g\in P_{\Theta }$, so
	 $wb_{\Theta ^{\ast }}$ and $w_{1}b_{\Theta ^{\ast }}$ belong to the 
	same orbit of $P_{\Theta }$. Reciprocally, if $wb_{\Theta ^{\ast }}$
	and $w_{1}b_{\Theta ^{\ast }}$ are in the same orbit of  $P_{\Theta }$
	then $\left( b_{\Theta },wb_{\Theta ^{\ast }}\right) $ and $\left(
	b_{\Theta },w_{1}b_{\Theta ^{\ast }}\right) $ belong to the the same orbit 
	of $G$.
	
\end{enumerate}

\begin{example}\label{productPn} If $F_{\Theta }$ is a projective space
	(real or complex) $\mathbb{P}^{n}$, then  $F_{\Theta ^{\ast }}$ is the
	Grassmannian $\mathrm{Gr}\left( n, n+1\right) $. In the language  of roots, 
	$\Theta $ is the complement of  $\{\alpha _{12}\}$ and  $\Theta ^{\ast }$
	is the complement of  $\{\alpha _{n,n+1}\}$.
	
	Taking the basis $\{e_{1},\ldots ,e_{n+1}\}$, in the canonical realization,  $%
	b_{\Theta }=\left[ e_{1}\right] $ whereas $b_{\Theta ^{\ast }}=\left[
	e_{1},\ldots ,e_{n}\right] $. An element  $w\in \mathcal{W}$ is a 
	permutation, so that $wb_{\Theta }$ (respectively $wb_{\Theta
		^{\ast }}$) is obtained from  $b_{\Theta }$ (respectively $b_{\Theta ^{\ast
	}}$) by permutation of the indices. For instance, $w_{0}\left[ e_{1}%
	\right] =\left[ e_{n+1}\right] $ and $w_{0}b_{\Theta ^{\ast }}=\left[
	e_{2},\ldots ,e_{n+1}\right] $ since $w_{0}$ inverts the order of the indices.
	
	In this case $P_{\Theta }$ is the group of $(n+1)\times (n+1)$ matrices of type 
	\begin{equation*}
		\left( 
		\begin{array}{cc}
			\ast _{1\times 1} & \ast  \\ 
			0 & \ast _{n\times n}%
		\end{array}%
		\right) .
	\end{equation*}%
	It has two orbits in  $\mathrm{Gr}\left(n, n+1\right) $. They 
	are:
	
	\begin{enumerate}
		\item The hyperplanes containing  $\left[ e_{1}\right] $, that is, the 
		orbit of $b_{\Theta ^{\ast }}=\left[ e_{1},\ldots ,e_{n}\right] $. 
		In fact, such a hyperplane is determined by its intersection with
		$\left[ e_{2},\ldots ,e_{n+1}\right] $ and the subgroup of matrices 
		\begin{equation*}
			\left( 
			\begin{array}{cc}
				\ast _{1\times 1} & 0 \\ 
				0 & \ast _{n\times n}%
			\end{array}%
			\right) 
		\end{equation*}%
		is   transitive already  in the  Grassmannian of subspaces of  $\dim =n-1$ in 
		$\left[ e_{2},\ldots ,e_{n+1}\right] $.
		
		\item The hyperplanes transversal to   $\left[ e_{1}\right] $, that is,  
		the orbit of $w_{0}b_{\Theta ^{\ast }}=\left[ e_{2},\ldots ,e_{n+1}\right] $.
		In fact, if $V$ is a hyperplane transversal to  $\left[ e_{1}\right] $ 
		then the matrix $g\in P_{\Theta }$ whose columns from  $2$ to  $n+1$ are
		the coordinates of a basis $\{v_{2},\ldots ,v_{n+1}\}$ of $V$ satisfies $g\left[
		e_{2},\ldots ,e_{n+1}\right] =V$. 
	\end{enumerate}
	
	In conclusion, the diagonal action of  $\mathrm{SL}\left( n+1,\ast \right) $ in $%
	\mathbb{P}^{n}\times \mathrm{Gr}\left(n, n+1\right) $ has two orbits, 
	an open one and a closed one. The open orbit is isomorphic to the adjoint orbit
	$\mathrm{Ad}\left( G\right) H_{\Theta }$ with $H_{\Theta }=%
	\Diag(n,-1,\ldots ,-1)$ and is formed by the pairs of transversal elements in
	$\mathbb{P}^{n}\times \mathrm{Gr}\left(n, n+1\right) $.
	On the other hand, the closed orbit is isomorphic to the flag $F_{\Theta \cap
		\Theta ^{\ast }}$. Since $\Theta \cap \Theta ^{\ast }$ is the complement of 
	$\{\alpha _{12},\alpha _{n,n+1}\}$ it follows that $F_{\Theta \cap
		\Theta ^{\ast }}=F\left( 1,n\right) $. 
\end{example}

\begin{example}\label{roots}Consider the Grassmanians 
 $F_{\Theta }=\mathrm{Gr}\left(k, n+1\right) 
	$  and $F_{\Theta ^{\ast }}=\mathrm{Gr}\left(n+1-k, n+1\right) $, real
	or complex (it is preferable to assume $k<(n+1)/2$). Here, in terms of roots,  $\Theta $ 
	is the complement of  $\{\alpha _{12},\ldots ,\alpha _{k-1,k}\}$ while $\Theta
	^{\ast }$ is the complement of  $\{\alpha _{k,k+1},\ldots ,\alpha _{n,n+1}\}$.
	In $\mathrm{Gr}\left(k, n+1 \right) \times \mathrm{Gr}\left(n+1-k, n+1\right) $
	there exist $k+1$ orbits of the diagonal action determined by the pairs 
	$\left( V,W\right) \in \mathrm{Gr}\left(k,
	n+1\right) \times \mathrm{Gr}\left(n+1-k, n+1\right) $ such that  $\dim \left(
	V\cap W\right) =0,1,\ldots ,k$. The orbit determined by  $\dim \left( V\cap
	W\right) =0$ is the open orbit (transversal pairs) whereas the closed orbit
	is given by  $\dim \left( V\cap W\right) =k$, that is
	$V\subset W$. This closed orbit is the  flag $F\left(
	k,n+1-k\right) $ given by  $F_{\Theta }$ with $\Theta $ the
	complement of  $\{\alpha_{k,k+1},\alpha _{n+1-k,n+1-k+1}\}$.
\end{example}

\begin{example} The choice $F_{\Theta }=F$  gives the maximal flag,
which is 
	self-dual (for any group). The orbits of  $P_{\Theta }=P$
	are the same as the orbit of  $N^{+}$ which give the 
	Bruhat decomposition. In this case the closed orbit is  $%
	F$ itself.  
\end{example}

\section{Partial extension of  the potential}
\label{LGn}
Now we wish to describe how to extend the potential of our  LG-models to their compactifications. Let $H= 
\mathrm{Diag}(\lambda_1, \dots, \lambda_{n+1})\in\mathfrak{sl}(n+1,\mathbb{C})$ be a 
regular element so that $\lambda_i$ are all distinct. 
We consider the Landau--Ginzburg models
$(\mathcal O_n, f_H)$  where $\mathcal O_n$ is the adjoint orbit of $H_0=\Diag(n,-1,\ldots,-1)$ as in definition \ref{orbit} and $f_H$ is the height function on the orbit as described in \cite[Thm.\thinspace 2.2]{GGSM1},
obtained by pairing with $H$ via the Cartan--Killing form.

\begin{notation}\label{lgn}
	We denote by $\mathrm{LG}_n$ the Landau--Ginzburg model $(\mathcal O_n, f_H)$.
\end{notation}

We are looking for a tame compactification 
$\overline{\mathrm{LG}}_n=(Z,w)$ such that $Z\setminus D= \mathcal O_n$ for some divisor $D$ 
together with 
a holomorphic  extension $w$ of the potential 
$f_H$. 
In this section, we accomplish an intermediate step of the construction, namely, that of 
describing an extension of $f_H$ to 
a rational map $R_H$  defined 
in codimension 2
on the compactification $\mathbb P^n\times \mathbb P^n$ 
as in example \ref{productPn}. We also verify that the critical points of $R_H$ coincide with 
those of $f_H$ outside of the indeterminacy locus $\mathcal I$ (definition \ref{ilocus}), once this is done we can then 
obtain a holomorphic extension after blowing up $\mathcal I$, which we will do in section \ref{holext}. 

In the case of $\mathrm{LG}_n$
our rational map is described in \cite[Sec.\thinspace7]{BGGSM} as 
\begin{equation}
	R_H:\mathbb{P}^{n}\times Gr(n, n+1)\rightarrow \mathbb{P}^1, 
\end{equation} 
\begin{equation}
	R_H([v],[\varepsilon]) = \frac{\tr((v\otimes \varepsilon) \rho(H))}{\tr(v\otimes \varepsilon)}=\frac{\sum_{i=1}^{n+1}{\lambda_i a_{i1} (\adj g)_{1i} }}{\sum_{i=1}^{n+1}{ a_{i1} (\adj g)_{1i} }}.
\end{equation}
On $\mathbb P^{n}\times \mathbb P^{n}$ with bihomogeneous coordinates $x_1,\dots ,x_{n+1},y_1,\dots ,y_{n+1}$:
$$R_H([x_1,\dots,x_{n+1}], [y_1,\dots,y_{n+1}]) = \left[\sum_{i=1}^{n+1}\lambda_i x_iy_i, \sum_{i=1}^{n+1} x_iy_i\right].$$
For the general case, we have the following definitions:

\begin{definition}
	\label{flag1n}
	The flag variety $\mathbb F(1,n)$ is the homogeneous manifold consisting of pairs $(l,\pi)$ of lines contained in 
	an $n$-dimensional subspaces of $\mathbb C^{n+1}$.
$\mathbb F(1,n)$ can be obtained as 
the subvariety of $\mathbb P^{n}\times \mathbb P^{n}$ defined by: 
$x_1y_1+\cdots +x_{n+1}y_{n+1} =0.$
 \end{definition}

\begin{notation}
	\label{ilocus}
	The indeterminacy locus $\mathcal{I}$ of the rational map  $R_H$ is given by:
	$$\mathcal  I:= \mathbb F(1,n)\cap \{\lambda _1x_1y_1+\cdots +\lambda _{n+1}x_{n+1}y_{n+1}=0\}.$$
\end{notation}

\begin{lemma} \label{smoothn} $\mathcal I$ is smooth and $R_H$ has no critical points on $\mathbb F(1,n) \setminus \mathcal I $. 
\end{lemma}

\begin{proof}
	The Jacobian matrix of $R_H$ 
	$$J=  \left(\begin{array}{cccccc}
		\lambda_1y_1&  \dots & \lambda_{n+1}y_{n+1} &  \lambda_1x_1&  \dots & \lambda_{n+1}x_{n+1}\\
		y_1& \dots & y_{n+1} & x_1& \dots & x_{n+1}
	\end{array}\right)
	$$ 
	has  rank 2 on $\mathbb F(1,n) \setminus \mathcal I $. In fact, 
	since all $\lambda_i$ are distinct and $x_i$ and $y_i$ are coordinates in projective spaces, 
	the only points when the matrix has rank 1 are ones
when $x_i = y_j = 0$ unless $i = j = k$ for some $k$, but these do not  belong to $\mathbb F$ and 
nor to $\mathcal I$ either, showing in particular that 
$\mathcal I$ is smooth.
\end{proof}

Furthermore,  $R_H$ has  critical points on  the coordinate points $(e_i,e_i)$, that is, 
those points $[x_1,\dots,x_{n+1}], [y_1,\dots,y_{n+1}]$ with only 2 nonzero coordinates $x_i, y_j$ with $i= j$,
these lie outside $\mathbb F(1,n)$ and correspond to the critical points of $f_H$. Lemma \ref{smoothn}
does not verify points of $\mathcal I$ where the rational map $R_H$  remains  undefined 
(we resolve this issue in the following section, where such  points are blown-up).

\section{Holomorphic extension of  the potential}\label{holext}

To extend the potential as a holomorphic map we need to blow up  $ \mathbb P^{n}\times \mathbb P^{n}$ 
along $\mathcal I$. We denote the result of the blowing up by $Z$, obtained as follows:

Take $\mathbb P^1$ with  homogeneous coordinates $[t,s]$.  The pencil $\{tg+sf_H\}_{t,s\in\mathbb C}$ induces a rational map $\mathbb P^{n}\times \mathbb P^{n}\dashrightarrow \mathbb P^1$ with $\mathcal I$ as its indeterminacy locus. Call $Z$ the
closure of  the graph of this
map. 
\vspace{-2.5cm}
\begin{figure}[H]
	\centering
	\includegraphics[width=0.5\linewidth]{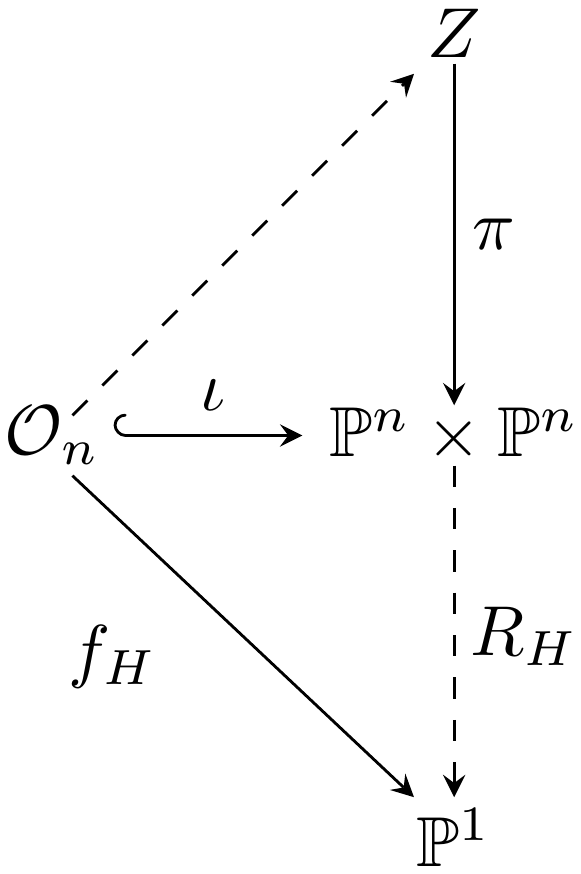}
	\label{fig:commutative-diagram-tikz-cd}
\end{figure}
\vspace{-2.5cm}
So, we get a map 
\begin{eqnarray*}
	w\colon Z \subset  \mathbb P^{n}\times \mathbb P^{n}\times \mathbb P^1 & \rightarrow &\mathbb P^1\\
	(x_1,\dots, x_{n+1}, y_1,\dots, y_{n+1}, t, s)&\mapsto&[t,s].
\end{eqnarray*}
Note that if $s\neq 0$ then 
\begin{equation} \label{onorbit} [t,s]=\left[\frac{t}{s},1\right]=\left[\frac{f_H}{g},1\right]= \left[f_H,g\right]
\end{equation}
where the middle equality holds since $tg=sf_H$, and $g=1$ over the orbit $\mathcal O_n$ 
because in fact $g=\det(A_{ij})$ with $A \in \SL(n+1, \mathbb C)$. Thus, we obtain the desired holomorphic extension:

\begin{lemma}\label{tame}
	The map $w\colon Z \rightarrow \mathbb P^1$ is a holomorphic  extension of $f_H$. The critical points of 
	$w$ coincide with the critical points of $f_H$. 
\end{lemma}

\begin{proof} 
	We know by lemma \ref{smoothn} that the indeterminacy locus $\mathcal I $ of $R_H$  is smooth. 
	We want to show that $w$ has no critical points over  the exceptional divisor $E$.  
	Since surjectivity of the derivative is a local question, 
	it is enough to analyse an open neighbourhood of the point in question (the analytic topology is sufficient, 
	though the Zariski topology will work as well). 
	
	For instance if we take a point  $p$ in $E$ mapping to   the  coordinate point $P=([1,0,\ldots ,0],[0,1,\ldots ,0]) \in \mathcal I$
	by the blow down map, 
	hence  $p= ([1,0,\ldots ,0],[0,1,\ldots ,0],[t_0,s_0])$. 
	We take the open neighbourhood 
	$$U_{12} = \{x_1\neq 0, y_2\neq 0\} \subset \mathbb P^n \times \mathbb P^n$$
	of  the point $P$. 
	In this neighbourhood, the defining equations for $S=\mathcal I \cap U_{12}$ become
	$$\left\{\begin{array}{rlll} f=&  \lambda_1Y_1+\lambda_2X_2+\lambda_3X_3Y_3+\dots+\lambda_{n+1}X_{n+1}Y_{n+1}&=&0\\
		g= & Y_1+X_2+X_3Y_3+\dots+X_{n+1}Y_{n+1}&=&0
	\end{array}\right . ,$$
	where $X_i=x_i/x_1, Y_i=y_i/y_2$,  hence the local expressions for $f$ and $g$ 
	are nonsingular. 
	Since $S$ is a smooth submanifold of $U_{12}\simeq \mathbb C^{2n}$
	of codimension 2, we can change coordinates  to $\Delta =\mathbb C[Z_1,\dots,Z_{2n}]$ so that 
	$$Z_{2n-1}=f, \quad Z_{2n}=g$$ 
	and therefore $S$  is cut out inside $\Delta$  as the linear submanifold $Z_{2n-1}=Z_{2n}=0$.
	Now, following \cite[p.\thinspace 603]{GH}, we know that an extension of $w$ is defined over $\tilde\Delta$ by $(w(z), l') $
	where $$l'=\left(\frac{\partial w}{\partial Z_{2n-1}}t, \frac{\partial w}{\partial Z_{2n}} s\right).$$
	Now consider the chart on the target $\mathbb P^1$ where 
	$s\neq 0$, then  on $\tilde \Delta \setminus E$ we have 
	$w=f/g=t/s$, with $s\neq 0$ hence $g\neq 0$  thus
	$$l'=\left(\frac{\partial f/g}{\partial f} t_0, \frac{\partial f/g}{\partial g} s_0\right)=(t_0/g,-s_0f/g^2).$$   
	So, that $w\vert_S$ can be extended over $E$ to 
	$$w\vert_S(z,l')= (w(z),(t_0/g,-s_0f/g^2)$$
	without critical points.
	We conclude that $P$ is not a critical point of $w$. 
	
	Calculations at other coordinate points work similarly, and 
	since every point on $E$ belongs to an 
	open neighbourhood of some coordinate point, 
	extension of the result to more general points 
	of $E$ goes through automatically.
	In all such cases we conclude surjectivity of the derivative, hence $w$ has no critical points on $E$.
\end{proof}

\begin{remark}\label{inffib}
Here  $[t,s]$ are the coordinates of $\mathbb P^1$ and 
the fibre over infinity $Z_\infty=w^{-1}([1,0])$ happens when 
$s=0$, giving the  equation
 \begin{equation}\label{sol} t \left(\sum x_iy_i\right)  = 0.\end{equation}
The set  $V$  of solutions of (\ref{sol})
 inside $\mathbb P^n \times \mathbb P^n \times  \mathbb P^1$ consists of 2 components
 \begin{itemize}
 \item[$\iota)$]  a  flag manifold 
$ \mathbb F(1,n) $ cut out by the equation  $\sum x_iy_i = 0$ in $\mathbb  P^n \times \mathbb P^n$
 for any value of $[t,s] \in \mathbb P^1$, thus forming a component that is a 
 $\mathbb P^1$ bundle over  $ \mathbb F(1,n) $, and  
\item[$\iota\iota)$] when $t=0$, we have a second component 
$\mathbb  P^n \times \mathbb P^n \times {[0,1]}$. 
 
 These 2 components  intersect in $\mathbb F(1,n) \times {[0,1]}$. 
  \end{itemize}
 
  The fibre at infinity $Z_\infty$
of the map $w\colon Z \rightarrow \mathbb P^1$ is formed by  the solutions of (\ref{sol}) inside $Z$, 
that is, $Z_\infty=Z\cap V$.
Since  the point at infinity is $[t,s]=[1,0]$, 
we can not at the same time have that $t=0$, so that the entire second component, described in item 
$\iota\iota)$ can not belong to  $Z_\infty $.
%
 Taking  $s=0$  we obtain precisely the proper transform 
 of the blow-up of $\mathcal I$ inside $Z$, and
we conclude that $Z_\infty =w^{-1}([1,0])\simeq \mathbb F(1,n) $.
\end{remark}

\begin{lemma}
$(Z,w)$ is a weakly tame compatification of $(\mathcal O_n,f_H)$.
\end{lemma}

\begin{proof} By construction $Z$ compactifies $\mathcal O_n$ and by Lemma \ref{tame}
$w$ is an extension of $f_H$.
 Writing the  boundary divisor $D_{Z}$ as a sum of 2 two irreducible  components $D^{v}$ and $D^{h}$, 
 we have that the vertical component $D^{v}\simeq \mathbb F(1,n)$  is a smooth irreducible divisor which is the fibre of the map 
 $w$ over infinity as described in Remark \ref{inffib}.
The horizontal component $D^h$ of $D^{Z}$ is the exceptional divisor of the blow up, therefore a $\mathbb P^1$- bundle over 
$\mathcal I$, 
where $\mathcal I \subset \mathbb{P}^{n}\times \mathbb{P}^{n}$ is the indeterminacy locus of the rational map $R_H$ 
as in Definition \ref{ilocus}. 

Therefore, $D_{Z} = D^{h} +  D^{v}$ 
is a strict normal crossing divisor, such that $D^{h}$ is smooth and surjective over $\mathbb{P}^{1}$
 and  $D^{v}$ is the scheme theoretic fibre  over infinity.  
 Since $\mathcal O_n$ has trivial canonical class, one can find a unique anticanonical divisor $-K_Z$ with support $D_{Z}$. 
 Such divisor fails the condition in Definition \ref{def-3} that requires that this anticanonical divisor be equal to $D_{Z}$;
 instead,  it is a combination of $D^{h}$ and $D^{v}$  with positive coefficients: 
$$-K_{Z} =  nD^{h} + (n+1) D^{v}$$
and because of these nontrivial multiplicities we refer to our compactification as weakly tame.
\end{proof}

\begin{remark} Given our choice of smooth compactification and the fact that the fibre at infinity is smooth, 
as explained in remark \ref{inffib}, we obtain that the monodromy around infinity 
is trivial, that is, the map (\ref{monodromy-2}) is the identity. Consequently, the grading 
of the cohomologies to be considered in definition \ref{hpq} is also trivial. However, notice that 
it does not imply that all the invariants $h^{p,q}(Y,w)$ are zero. Indeed, 
       $$h^{n,n}(Y,w)=\dim \HH^{2n}(Y,Y_b) = n+1$$
       by lemma \thinspace\ref{middle}.
\end{remark}

\section{Hodge structures on  minimal adjoint orbits}
The following results will be used to 
show that the adjoint orbits have pure Hodge structures. Consider $X$ smooth projective, and $Y$ smooth projective of codimension 1 in $X$. 
To describe the Hodge structure 
on $U = X \setminus Y$, we consider the
Gysin sequence 
\begin{equation}\label{gysin}  
	\ldots\rightarrow \HH^{k-2} (Y) \rightarrow \HH^k(X) \rightarrow \HH^k(U)\rightarrow \HH^{k-1}(Y) \rightarrow \ldots .\end{equation}

It turns out that to make the  homomorphism $\HH^{k-2}(Y)\rightarrow \HH^k(X)$  a map of Hodge structures, it is sufficient 
to shift the weights up by (1,1) \cite[Sec.\thinspace 2.1]{Ho}. So we get the Gysin map
$$\delta_k\colon \HH^{k-2}(Y)\otimes \mathbb Q(-1)\rightarrow \HH^k(X)$$
which is a map of 
Hodge structures of weight $k$, and thus induces Hodge structures on $\ker \delta_k$ and $\coker \delta_k$, 
see \cite[Thm.\thinspace10]{Ho}.
Furthermore, the Hodge filtrations on $\ker \delta_{k+1}$ and $\coker \delta_k$ are the same as those induced by 
the Hodge filtration on $\HH^k(U)$ via the short exact sequence (arising from the Gysin sequence) 
$$ 0 \rightarrow \coker \delta_{k} \rightarrow \HH^k(U) \rightarrow \ker \delta_{k+1} \rightarrow 0 \textup{.}$$  

Thus,  $\HH^n(U)$ admits a natural mixed Hodge structure with weight filtration 
$$W^k\HH^n(U) = 
\left\{\begin{array}{ll}
	0 & k<n,\\
	\imagen \HH^n(X)\rightarrow \HH^n(U) & k=n,\\
	\HH^n(U) & k>n,
\end{array}
\right.$$
and Hodge filtration $F^p\HH^n(U)$ given by classes represented by $\geq $ $p$-holomorphic logarithmic differential forms such that 
$$Gr^k\HH^n(U) = 
\left\{\begin{array}{ll}
	0 & k<n, k>n+1,\\
	\coker \delta_n & k=n,\\
	\ker \delta_{n+1} & k=n+1.
\end{array}
\right.$$
where the kernel and cokernel of $\delta_k$ are given  their natural Hodge structures \cite[Cor.\thinspace 11]{Ho}.

\addtocontents{toc}{\setcounter{tocdepth}{-10}}

\subsection{The Hodge structure on $\mathcal O_n$} \label{pp1}
We consider  the case when  $U = \mathcal O_n$ is the minimal adjoint orbit  as in definition \ref{orbit}
compactified to 
$\mathbb P^n \times {\mathbb P^n}$ so that $\mathbb P^n \times {\mathbb P^n} \setminus U = \mathbb F(1,n)$ is a flag variety. 

\begin{example} \label{pure} Taking $U = \mathcal O_n$ and $X= \mathbb P^n \times \mathbb P^n$ in (\ref{gysin})
	we get  
	$$\rightarrow \HH^{k-2} (\mathbb F(1,n)) \rightarrow \HH^k(\mathbb P^n\times \mathbb P^n) \rightarrow \HH^k(\mathcal O_n)
	\rightarrow \HH^{k-1}(\mathbb F(1,n)) \rightarrow \, \text{,}$$
	and  obtain a trivial weight filtration 
	$$W^k\HH^n(\mathcal O_n) = 
	\left\{\begin{array}{ll}
		0 & k<n,\\
		\HH^n(\mathcal O_n) & k\geq n.
	\end{array}
	\right.$$
	Thus, the Hodge structures on  the minimal adjoint orbits $\mathcal O_n$ are pure.
\end{example}

We show that our $LG$ models are Hodge--Tate.
Let $Z=\mathrm{Bl}_{\mathcal I}{\mathbb P^n\times\mathbb P^n}$
as in section \ref{holext}
with $E$  the exceptional divisor. So, we have 
that $Z\backslash E=\mathbb P^n\times\mathbb P^n\backslash\mathcal I$. We  compute the Hodge numbers of $E$.

\begin{lemma}\label{exc} The Hodge numbers of the exceptional divisor $E$ are $h^{p,q}(E)=0$ unless $p=q$.
\end{lemma}
\begin{proof}
	Firstly, observe that $\mathbb F(1,n)$ is a $\mathbb P^1$-bundle over $\mathcal I$. Secondly, 
	observe that $E$ is a $\mathbb P^1$-bundle over $\mathcal I$. So we have the Hodge polynomials
	$$\alpha(E) = \alpha(\mathbb F(1,n)) .$$
	
	On the other hand $\mathbb F(1,n)$ fibres over $\mathbb P^n$ with fibre $\mathbb P^{n-1}$. 
	Thus, 
	$ \alpha(E)=(1+uv+\cdots +u^{n-1}v^{n-1} )(1+uv+\cdots+u^nv^n).$ 
	
	Thus, we conclude that $E$ has only $(p,p)$ cohomology.
\end{proof}

\begin{lemma}\label{pneqq} If $p\neq q$ then the KKP numbers of $LG_n$ vanish, that is, $f^{p,q}(Y,w)=h^{p,q}(Y,w)=i^{p,q}(Y,w)=0$.
\end{lemma}

\begin{proof}
	Example \ref{pure} describes the trivial weight filtration on the orbit $Y=\mathcal O_n$ 
	showing that the Hodge structure of $\mathcal O_n$ is pure. 
	By  lemma \ref{exc} we have that the Hodge numbers of 
	$E$ vanish for $p\neq q$,  and similarly for the compactification $Z$ of $\mathcal O_n $ obtained by blowing up 
	$\mathbb P^n \times \mathbb P^n$ at $\mathcal I$. Therefore,  the same is also true for orbit $\mathcal O_n$. 
	Lemma \ref{tame} shows that the compactified Landau--Ginzburg model  
	$w\colon Z \rightarrow \mathbb P^1$ has no critical points  over $E$, implying in particular that is does not acquire extra critical points at infinity, see remark \ref{inffib}.
	It then follows that the logarithm of the monodromy operator at infinity $N$ is trivial for all cohomology groups, 
	forcing the first Hodge theoretical invariants, $h^{p,q}(Y,w)$, to vanish when $p\neq q$.  
	
	The numbers $f^{p,q}(Y,w)$ can be computed from the dimensions of the Hodge graded pieces
	of the canonical  mixed 
	Hodge structure on the relative cohomology $\HH^m(Y,Y_b)$ using 
	results from \cite{Ha} or \cite{Sh}. For the cases we consider here, 
	the equality $f^{p,q} (Y,w)= h^{p,q} (Y,w)$ 
	follows from \cite[Thm.\thinspace 3.1]{Ha},
	which shows that the $f^{p,q}$ Hodge numbers are computed 
	from the dimensions of the Hodge graded pieces of the pair $(Y,w)$,
	implying that $f^{p,q} (Y,w)= h^{p,q}(Y,w) $.

	 Since  our compactified   Lefschetz fibration $w \colon Z \rightarrow \mathbb P^1$ has no critical points at infinity, 
it 	has the same critical points as $f$, see lemma \ref{tame}. Consequently, by standard results for  vanishing cycles in 
	Lefschetz theory, we know that there exist $n+1$ vanishing cycles in middle cohomology for $w$, 
	one for each of its critical points. Furthermore, we also know that  the relative cohomologies  $\HH^m(Y,Y_b;\mathbb C)$ vanish for $m\neq \frac{1}{2}\dim Y$, and 
	$\HH^{2n}(Y,Y_b;\mathbb C) = n+1$ where $2n$ equals  half the real  dimension of $Z$ (see lemma \ref{middle} for an alternative proof).
	To complete the proof showing that the third invariant, 
	$i^{p,q}(Y,w)=0$  unless  $p=q=n$, we work in the opposite direction.  Katzarkov--Kontsevich--Pantev \cite{KKP} proved the equality  $$\HH^{m}(Y,Y_b;\mathbb C) =\sum_{p+q=m} i^{p,q}(Y,w) .$$
	Therefore, it remains  to verify that
	$i^{n,n}(Y,w)=n+1$.
	Let $\lambda_i$ be a critical value of $w$. Since $w$ is a Lefschetz fibration and $\dim Z=2n$, the potential
	$w$ can be written in the form $z_1^{2}+ \dots +z_{2n}^2$ around the critical point.
	Hence, the monodromy around the singular fibre $Y_{\lambda_i}$ is a Dehn twist  with local monodromy given by 
	a matrix $T= \begin{pmatrix}1 & 1 \\ 0 & 1\end{pmatrix}$ applied to local coordinates of $\mathbb C^n \oplus \mathbb C^n$. 
	Consequently,  $N= \log T =\begin{pmatrix}0 & 1 \\ 0 & 0\end{pmatrix} $ and $N^2=0$. 
	It then follows that the weight filtration corresponding to this monodromy around  $Y_{\lambda_i}$  has only one step, and the critical point with critical value $\lambda_i$ 
	contributes with  a $+1$ towards  the value of $i^{n,n}(Y,w)$.  Summing over all critical points we then get $i^{n,n}(Y,w)=n+1$. 
\end{proof}
\addtocontents{toc}{\setcounter{tocdepth}{1}} 
\begin{lemma}\label{middle}
	$ \displaystyle h^{2k}(Y,Y_b;\mathbb C)= \left\{ \begin{array}{ll} k+1 &  \textup{if } k=n,\\  
		0 & \textup{otherwise}. \end{array}\right. $
\end{lemma}

\begin{proof}  Using  \cite[Thm.\thinspace 2.1]{GGSM2}, we have that 
	$Y= \mathcal O_n \sim \mathbb P^n $ ($\sim$ denotes homotopy equivalence) and  using \cite[Cor.\thinspace 3.4]{GGSM1} and the fact that $f_H$ has $n+1$ critical points, we get that 
	$Y_b \sim \mathbb P^n \setminus \{P_1,\ldots, P_{n+1}\}$ where $P_1,\ldots, P_{n+1}$ are  points in 
	$\mathbb P^n$. 
	Using the Mayer--Vietoris sequence we obtain
	$$
	\HH^i(Y_b;\mathbb{C})=\begin{cases}
		\mathbb{C}&   i \textup{ even, }i<2n-1,\\
		\mathbb C^{\oplus n+1}& i=2n-1,\\
		0&  \textup{otherwise}.
	\end{cases}$$
	Then, the long exact sequence of the pair
	$$ \ldots \rightarrow \HH^n(Y,Y_b;\mathbb{C})\stackrel{j^\ast}{\rightarrow} \HH^n(Y;\mathbb{C})\stackrel{i^\ast}{\rightarrow} \HH^n(Y_b;\mathbb{C})\stackrel{\delta}{\rightarrow}\HH^{n+1}(Y,Y_b;\mathbb{C})\rightarrow \ldots $$
	gives
	$$\HH^i(Y,Y_b;\mathbb{C})=\begin{cases}
		\mathbb C^{\oplus n+1}& i=2n,\\
		0&  \textup{otherwise.}
	\end{cases}$$
\end{proof}

\begin{theorem}\label{main}
	The LG model $(\mathcal O_n,f_H)$ admits a weakly tame compactification and satisfies the KKP conjecture.
\end{theorem}
\begin{proof} Lemmas \ref{pneqq} and \ref{middle} together with the equality 
	(\ref{sums})
	imply that the three invariants coincide on $\mathcal O_n$. 
	We depict their values in  the KKP diamond:
	\begin{figure}[ htb ]
		\centering
		 $\begin{array}{lcccccccccccr}
			
			& && && &0& && && &\\
			& && && && && && &\\
			&  & & & & 0 & & 0 & &  & &  &\\
			& & & & & &\tiny\vdots& & & & & & \\
			& &  & & 0 &\tiny{\cdots} &n+1&\tiny\cdots &0 & & & &\\
			&\ \ & &\ \ & &\ \ &\vdots&\ \ & &\ \ & &\ \ &\\
			&& && &0& &0& && &&\\
			& && && && && && &\\
			& && && &0& && && &
			
		\end{array}$
	\end{figure}
	
\end{proof}

\backmatter

 \bmhead{Acknowledgements}
\noindent We thank the referee for many useful suggestions.
Ballico was  partially supported by MIUR and GNSAGA of INdAM (Italy).
Gasparim was partially supported by  the Vicerrector\'ia de Investigaci\'on y Desarrollo Tecnol\'ogico de la  Universidad Cat\'olica del Norte (Chile).
Rubilar acknowledges support from  Beca Doctorado Nacional -- Folio 21170589. 
Gasparim and Rubilar were partially supported by Network NT8 of the Office of External Activities of ICTP (Italy).

\bmhead{Data availability}
Data sharing not applicable to this article as no datasets were generated or analysed during the current study.

\bibliography{sn-bibliography}

\end{document}